\documentclass[10pt]{article}

\usepackage{amsmath}
\usepackage{amsthm}
\usepackage{amssymb}
\usepackage{amsfonts}
\usepackage{graphicx}
\usepackage{verbatim}
\usepackage{mathrsfs}
\usepackage{subfigure}
\usepackage{fancyhdr}
\usepackage{latexsym}
\usepackage{graphicx}
\usepackage{dsfont}
\usepackage{epsf} 
\usepackage{url}
\usepackage{enumerate}

\theoremstyle{plain}
\newtheorem{theorem}{Theorem}[section]
\newtheorem{proposition}[theorem]{Proposition}
\newtheorem{lemma}[theorem]{Lemma}

\newtheorem{remark}[theorem]{Remark}
\newtheorem{definition}[theorem]{Definition}

\newcommand{\R}{\mathbb{R}}
\newcommand{\N}{\mathbb{N}}
\newcommand{\C}{\mathcal{C}}

\newcommand{\constC}{\mathcal{K}}
\DeclareMathOperator{\sgn}{sgn}

\author{Maxime Breden \thanks{CMLA, ENS Cachan, CNRS, Universit\'e Paris-Saclay, 94235 Cachan, France. {\tt breden@cmla.ens-cachan.fr}}}

\begin{document}

\title{Applications of improved duality lemmas to the discrete coagulation-fragmentation equations with diffusion}

\date{\today}

\maketitle

\begin{abstract}
In this paper, we investigate the use of so called "duality lemmas" to study the system of discrete coagulation-fragmentation equations with diffusion. When the fragmentation is strong enough with respect to the coagulation, we show that we have creation and propagation of superlinear moments. In particular this implies that strong enough fragmentation can prevent gelation even for superlinear coagulation, a statement which was only known up to now in the homogeneous setting. We also use this control of superlinear moments to extend a recent result from \cite{BreDevFel16}, about the regularity of the solutions in the pure coagulation case, to strong fragmentation models.
\end{abstract}

\begin{center}

\medskip

\end{center}

\section{Introduction}\label{sec:intro}

In this work we consider the diffusive coagulation-fragmentation system describing the dynamics of clusters coalescing to build larger clusters and breaking apart into smaller pieces. Coagulation models were first introduced by Smoluchowski (see \cite{Smo16,Smo17}) and then complexified to take into account other effects like fragmentation and diffusion. These models are used in numerous and diverse applications at very different scales, in physics (smoke, sprays), chemistry (polymers), or biology (hematology, animal collective behavior). 

In this work we consider only discrete (in size) models, i.e. we assume that clusters can be of size $i\in\N^*$, and we denote by $c_i=c_i(t,x)$ the concentration of clusters of size $i$ at time $t$ and position $x$. We also assume that the clusters are confined in a smooth bounded domain $\Omega$ of $\R^N$. For any positive time $T$, we denote by $\Omega_T$ the set $[0,T]\times\Omega$. The concentrations $c_i$ satisfy the following set of equations, for all $i\in\N^*$,
\begin{equation}
\label{eq:syst_coag-frag}
\left\{
\begin{aligned}
& \partial_t c_i - d_i \Delta_x c_i = Q_i(c) + F_i(c), \quad &\text{on}& \ \Omega_T, \\
& \nabla_x c_i\cdot \nu = 0 \quad &\text{on}& \ [0,T]\times\partial\Omega, \\
& c_i(0,\cdot) = c_i^{in} \quad &\text{on}& \ \Omega,
\end{aligned}
\right.
\end{equation}
where $\nu(x)$ is a unit normal vector at $x\in\partial\Omega$ and the initial concentrations $c_i^{in}\geq 0$ are given. The coagulation terms $Q_i(c)$ and the fragmentation terms $F_i(c)$ respectively write:
\begin{equation}
\label{def:Q}
Q_i(c)=Q_i^+(c) - Q_i^-(c) = \frac{1}{2}\sum_{j=1}^{i-1}a_{i-j,j}c_{i-j}c_j - \sum_{j=1}^{\infty}a_{i,j}c_ic_j,
\end{equation}
\begin{equation}
\label{def:F}
F_i(c)=F_i^+(c) - F_i^-(c) = \sum_{j=1}^{\infty}B_{i+j}\beta_{i+j,i}c_{i+j} - B_ic_i.
\end{equation}
The nonnegative parameters $B_i$, $\beta_{i,j}$ and $a_{i,j}$ represent the total rate $B_i$ of fragmentation of clusters of size $i$, the average number $\beta_{i,j}$ of clusters of size $j$ produced due to fragmentation of a cluster of size $i$, and the coagulation rate $a_{i,j}$ of clusters of size $i$ with clusters of size $j$. The fragmentation of one cluster into smaller pieces should conserve mass, clusters of size $1$ should not fragment further and the coagulation rates should be symmetric, so for all $i,j\in\N^*$, we impose 
\begin{equation}
\label{hyp:modele}
i=\sum_{j=1}^{i-1}j\beta_{i,j}, \quad B_1=0,\quad  a_{i,j}=a_{j,i}  \quad \text{and} \quad a_{i,j},B_i,\beta_{i,j}\geq 0. 
\end{equation}
For more details on both the modeling and the applications, we refer the reader to the surveys \cite{Dra72,LauMis04,DevFel13} and the references therein. 

Assumption~\eqref{hyp:modele} allows us to write (formally for any sequence $(\varphi_i)$) the weak formulation:
\begin{equation}
\label{eq:form_faible_Q}
\sum_{i=1}^{\infty} \varphi_iQ_i(c) = \frac{1}{2}\sum_{i=1}^{\infty}\sum_{j=1}^{\infty}a_{i,j}c_ic_j(\varphi_{i+j}-\varphi_i-\varphi_j),
\end{equation}
\begin{equation}
\label{eq:form_faible_F}
\sum_{i=1}^{\infty} \varphi_iF_i(c) = -\sum_{i=2}^{\infty}B_ic_i\left(\varphi_i-\sum_{j=1}^{i-1}\beta_{i,j}\varphi_j\right).
\end{equation}
For any $k\geq 0$, we define the moment of order $k$ (associated to solutions $(c_i)$ of \eqref{eq:syst_coag-frag}-\eqref{hyp:modele}) as
\begin{equation*}
\rho_k(t,x)=\sum_{i=1}^{\infty}i^kc_i(t,x),
\end{equation*}
and similarly the initial moment of order $k$ as
\begin{equation*}
\rho_k^{in}(x)=\sum_{i=1}^{\infty}i^kc_i^{in}(x).
\end{equation*}
Considering $\varphi_i=i$ in the weak formulation \eqref{eq:form_faible_Q}-\eqref{eq:form_faible_F}, it is easy to see that (at the formal level), the total mass $\int\limits_{\Omega}\rho_1$ is conserved. Indeed we get that
\begin{equation}
\label{eq:rho_1}
\partial_t \left(\sum_{i=1}^{\infty} ic_i\right) - \Delta_x\left(\sum_{i=1}^{\infty} id_ic_i\right) =0,
\end{equation}
and an integration by part yields (thanks to the homogeneous Neumann boundary conditions)
\begin{equation}
\label{eq:non_gelation}
\frac{d}{dt} \int\limits_{\Omega}\rho_1(t,x)dx = 0 \quad \text{and thus}\quad  \int\limits_{\Omega}\rho_1(t,x)dx=  \int\limits_{\Omega}\rho_1^{in}(x)dx,\quad \forall t\geq 0.
\end{equation} 
Let us point out that \eqref{eq:non_gelation} can fail to be true if the coagulation is strong enough, the mass $\int\limits_{\Omega}\rho_1$ then becoming strictly smaller than the initial mass $\int\limits_{\Omega}\rho_1^{in}$ after some finite time $t^*$. This phenomenon is called gelation (in the homogeneous case see for instance \cite{HenErnZif83}, or \cite{EscMisPer02} for continuous (in size) models).

In all cases, weak solutions satisfy
\begin{equation*}
\int\limits_{\Omega}\rho_1(t,x)dx \leq  \int\limits_{\Omega}\rho_1^{in}(x)dx,\quad \forall t\geq 0.
\end{equation*} 
which gives a first \textit{a priori} estimate stating that the mass $\rho_1$ lies in $L^1(\Omega_T)$. 

Before proceeding further, let us introduce a precise definition of weak solutions, following~\cite{LauMis02}, that we will use throughout this paper.
\begin{definition}\label{def:sol_faible}
A global weak solution $c=\left(c_i\right)_{i\in\N^*}$ to \eqref{eq:syst_coag-frag}-\eqref{hyp:modele} is a sequence of nonnegative functions $c_i:[0,+\infty)\times\Omega\to[0,+\infty)$ such that, for all $i\in\N^*$ and $T>0$
\begin{itemize}
\item $c_i\in \mathcal{C}\left([0,T];L^1(\Omega)\right)$,
\item $Q^-_i(c),F^+_i(c)\in L^1(\Omega_T)$,
\item $\sup\limits_{t\geq 0}\int\limits_{\Omega}\rho_1(t,x)dx \leq \int\limits_{\Omega}\rho_1^{in}(x)dx$,
\item $c_i$ is a mild solution to the $i$-th equation in~\eqref{eq:syst_coag-frag}, that is
\begin{equation*}
\label{eq:def_ci_integral}
c_i(t)=e^{d_iA_1t}c_i^{in} + \int_0^t e^{d_iA_1(t-s)}\left(Q_i(c(s))+F_i(c(s))\right)ds,
\end{equation*}
where $Q_i$ and $F_i$ are defined by~\eqref{def:Q} and \eqref{def:F}, $A_1$ is the closure in $L^1(\Omega)$ of the unbounded, linear, self-adjoint operator $A$ of $L^2(\Omega)$ defined by
\begin{equation*}
D(A)=\left\{w\in H^2(\Omega),\ \nabla w\cdot \nu=0 \emph{ on }\partial\Omega \right\}, \qquad Aw=\Delta w,
\end{equation*}
and $t\mapsto e^{d_iA_1t}$ is the $\mathcal{C}^0$-semigroup generated by $d_iA_1$ in $L^{1}(\Omega)$.
\end{itemize}
\end{definition}
In \cite{LauMis02}, existence of weak solutions to \eqref{eq:syst_coag-frag}-\eqref{hyp:modele} was proven under the following assumption on the asymptotic behavior of the coagulation and fragmentation coefficients:
\begin{equation}
\label{hyp:exist_LM}
\lim\limits_{j\to\infty} \frac{a_{i,j}}{j} = 0 = \lim\limits_{j\to\infty} \frac{B_{i+j}\beta_{i+j,i}}{i+j}, \quad \forall~i\in\N^*.
\end{equation}
This existence result relies on a sequence of truncated versions of the system \eqref{eq:syst_coag-frag}-\eqref{hyp:modele}, for which existence of global smooth solutions is known. Some compactness argument is then used to extract a solution of the full system~\eqref{eq:syst_coag-frag}-\eqref{hyp:modele}. We detail this procedure in Section~\ref{sec:preuve_frag_forte} (see also \cite{LauMis02,Wrz97,Wrz04}).

Within the existence framework of \cite{LauMis02}, equation \eqref{eq:rho_1} was rewritten in \cite{CanDevFel10}, as
\begin{equation}
\label{eq:rho_1_M}
\partial_t \rho_1 - \Delta_x\left(M_1\rho_1\right) =0, \quad \text{where}\quad M_1=\frac{\sum_{i=1}^{\infty} i d_ic_i}{\sum_{i=1}^{\infty} i c_i},
\end{equation}
and duality techniques were then used to get another \textit{a priori} estimate on the mass $\rho_1$, namely that it lies in $L^2(\Omega_T)$, provided that
\begin{equation*}
\inf\limits_{i\in\N^*}d_i>0 \quad \text{and}\quad \sup\limits_{i\in\N^*}d_i <\infty,
\end{equation*}
and that the coagulation coefficients are strictly sublinear, i.e.
\begin{equation}
\label{eq:sub_lin}
a_{i,j}\leq C(i^{\alpha}j^{\beta}+i^{\beta}j^{\alpha}), \quad \text{with } \alpha+\beta<1, \quad \forall i,j\in\N^*.
\end{equation}
Thanks to new duality estimates from \cite{CanDevFel14}, it was recently proven in \cite{BreDevFel16}, in the pure coagulation case and still assuming~\eqref{eq:sub_lin}, that moments $\rho_k$ of any order in fact lie in every $L^p(\Omega_T)$, $p<+\infty$, if the initial moments $\rho_0^k$ lie in every $L^p(\Omega)$ and under some additional assumption (see~\eqref{hyp:convergence_di}) on the diffusion rates $d_i$. For other applications of these "duality estimates", see the survey~\cite{Pie10} and the references therein. 

We mention that similar results of propagation of moments were also obtained in \cite{Rez10,Rez14}, with different techniques and a slightly different set of hypothesis. The main result of \cite{BreDevFel16} is then to deduce propagation of smoothness of the solutions of \eqref{eq:syst_coag-frag}-\eqref{hyp:modele} (Theorem~\ref{th:all_moments_w} with no fragmentation) from these estimates in $L^p(\Omega_T)$ on the mass $\rho_1$. 
\medskip

In this paper, we investigate other consequences of the $L^p$ estimates of \cite{BreDevFel16}, this time in presence of fragmentation. Our main theorem states that, when the fragmentation is strong enough compared to the coagulation, we have creation and propagation of superlinear moments. This allows us to deduce that strong enough fragmentation can prevent gelation. To put this result in perspective, let us give a brief review (which is far from being exhaustive) of some of the main known results concerning the mass conserving solutions and the occurrence of gelation. For the homogeneous case, as long as the coagulation coefficients are sublinear, i.e.
\begin{equation*}
a_{i,j}\leq C(i+j), \quad \forall i,j\in\N^*,
\end{equation*}
the solutions of \eqref{eq:syst_coag-frag}-\eqref{hyp:modele} are mass conserving, that is \eqref{eq:non_gelation} rigorously holds (see for instance \cite{BalCar90}). In the inhomogeneous case, it was shown in \cite{CanDevFel10} that mass conservation still holds at least for strictly sublinear coagulation coefficients \eqref{eq:sub_lin}, the limit case $a_{i,j}=i+j$ still being open. However, it is known in the homogeneous case that gelation occurs as soon as the coagulation coefficients are strictly superlinear if there is no fragmentation (see \cite{HenErnZif83} for instance). Still in the homogeneous case, it was proven that having strong enough fragmentation could prevent this gelation phenomenon and ensure the conservation of mass even for superlinear coagulation coefficients (see \cite{Car92,Cos95}, and also \cite{EscLauMisPer03} for the continuous case). Here we prove that the same result holds for the inhomogeneous system with diffusion, under an additional assumption on the diffusion rates $d_i$:
\begin{equation}
\label{hyp:convergence_di}
d_i>0,\ \forall~i\in\N^* \quad \text{and}\quad d_i\underset{i\to \infty}{\longrightarrow}d_{\infty}>0.
\end{equation}
Here is the precise statement of our result.
\begin{theorem}\label{th:strong_frag}
Let $\Omega$ be a smooth bounded domain of $\R^N$, $N\leq 2$. Assume the following behavior for the coagulation and fragmentation coefficients:
\begin{equation}
\label{hyp:coef_frag_forte}
a_{i,j} \leq C_Q\left(i^{\alpha}j^{\beta}+i^{\beta}j^{\alpha}\right) \quad \text{and}\quad B_i\geq C_Fi^{\gamma},
\end{equation}
where $C_Q,C_F>0$, $0\leq\alpha,\beta\leq 1$, $\gamma\geq 1$ and $\alpha+\beta<\gamma$. Assume also \eqref{hyp:convergence_di}, together with 
\begin{equation}
\label{hyp:exist_frag_forte}
K_i^Q:=\sup\limits_{j\in\N^*}\frac{a_{i,j}}{j} <+\infty, \quad K_i^F:=\sup\limits_{j\in\N^*} \frac{B_{i+j}\beta_{i+j,i}}{i+j} <+\infty,\quad \forall~i\in\N^*.
\end{equation}
Finally, assume that $c_i^{in}$ lies in $L^{\infty}(\Omega)$ for all $i\geq 1$, that for some $k>1$ also satisfying $k>2-(\gamma-\alpha)$ and $k>2-(\gamma-\beta)$, $\rho_k^{in}$ lies in $L^1(\Omega)$, and that $\rho_1^{in}$ lies in $L^p(\Omega)$ for all $p<\infty$.
\par
Then there exists a weak solution to the coagulation-fragmentation system \eqref{eq:syst_coag-frag}-\eqref{hyp:modele}, whose moments satisfy
\begin{equation}
\label{eq:moments_strong_frag}
\int_0^T t^{m-1}\int\limits_{\Omega}\rho_{k+m(\gamma-1)}(t,x)dtdx < \infty,\quad\forall~m\in\N^*.
\end{equation}
In particular, the superlinear moment $\rho_{k+\gamma-1}$ lies in $L^1(\Omega_T)$, and thus gelation can not occur.
\end{theorem}

Before proceeding further, let us make a few comments about the hypothesis used in Theorem~\ref{th:strong_frag}.
\begin{itemize}
\item We point out that assumption~\eqref{hyp:exist_frag_forte} is more general than assumption~ \eqref{hyp:exist_LM}, which does not hold when $\alpha=1$ or $\beta=1$ in~\eqref{hyp:coef_frag_forte}. Therefore the existence of weak solution could not be obtained by simply applying the results of of~\cite{LauMis02}, and we had to develop new existence results (see Theorem~\ref{th:existence_gene} and Proposition~\ref{prop:existence}). Also the part about the coagulation in \eqref{hyp:exist_frag_forte} is in fact implied by \eqref{hyp:coef_frag_forte}.
\item It will be made apparent in the proof of Theorem~\ref{th:strong_frag} that the creation of moments displayed in~\eqref{eq:moments_strong_frag} is obtained through an iterative procedure (on $m$), and that assuming $\rho_1^{in}\in L^p(\Omega)$ for all $p<\infty$ is needed only if one want to get~\eqref{eq:moments_strong_frag} for all $m\in\N^*$. However, just assuming that $\rho_1^{in}\in L^p(\Omega)$ for a given $p$ can be enough to get at least a bound on one superlinear moment, and thus ensure that gelation can not happen. For instance, if $\rho_1^{in}$ lies in $L^p(\Omega)$ for $p=1+\frac{\gamma+k-2}{\gamma-(\alpha+\beta)}$, $p>2$, then the proof of Theorem~\ref{th:strong_frag} shows that $\rho_{k+\gamma-1}\in L^1(\Omega_T)$. 
\item It is also worth mentioning that in the case where $\rho_1^{in}$ lies in $L^p(\Omega)$ for a given $p$, the restrictions $N\leq 2$ and $c_i^{in}\in L^{\infty}(\Omega)$ for all $i\geq 1$ can be dropped, provided than \eqref{hyp:convergence_di} is replaced by a different \textit{closeness} assumption on the diffusion rates $d_i$ (depending on $p$) which we will introduce later \eqref{hyp:closeness}. This will be explained in Section~\ref{sec:preuve_frag_forte} along the proof of Theorem~\ref{th:strong_frag} (see Remark~\ref{rem:assumption_th2}).
\item Finally, we point out that our assumption $\alpha+\beta<\gamma$, which prescribes how strong the fragmentation should be compared to the coagulation, is more restrictive than the one made in \cite{Cos95} for the homogeneous case, where only $\alpha+\beta-1<\gamma$ is assumed. However, this difference is not related to the fact that our model includes spatial inhomogeneity, but comes from the fact that \cite{Cos95} only studies the case of binary fragmentation (where one cluster only breaks into exactly two smaller ones) and makes an additionnal assumption on the \emph{distribution of the fragmentation} (represented in our cases by the coefficients $\beta_{i,j}$). If we also make an additionnal assumption on the $\beta_{i,j}$, similiar to the one of~\cite{Cos95}, we can show that Theorem~\ref{th:strong_frag} still holds with the weaker assumption $\alpha+\beta-1<\gamma$ (more details after the proof, in Remark~\ref{rem:beta}). 
\end{itemize}

\medskip
As we pointed out earlier, to treat the stong fragmentation case we have to develop new existence results (see Proposition~\ref{prop:existence}). In the process, we obtain a new proof of existence of weak solutions of~\eqref{eq:syst_coag-frag}-\eqref{hyp:modele} that is also valid in many non strong framentation cases. This proof was already presented in \cite{CanDevFel10} but only in dimension 1. Here the improved duality estimates allow us to treat any dimension. Compared to the existence result of \cite{LauMis02}, our theorem requires an additional assumption~\eqref{hyp:di} on the diffusion coefficients and a more stringent assumption on the initial data ($\rho_1^{in}\in L^{2+\varepsilon}(\Omega)$ instead of $\rho_1^{in}\in L^{1}(\Omega)$), but we then get that $\rho_1$ lies in $L^{2+\varepsilon}(\Omega_T)$. Our proof has the advantage of being rather simple, and its other virtue is that it can be adapted for a variation of the model presented here, where fragmentation is induced by binary collisions between clusters (see \cite{CanDevFel10} and the references therein), leading to new quadratic terms that do not seem to be easy to treat with the inductive approach of \cite{LauMis02}. The precise statement of our existence result is the following:

\begin{theorem}
\label{th:existence_gene}
Let $\Omega$ be a smooth bounded domain of $\R^N$. Assume that the coagulation and fragmentation coefficients satisfy the asymptotic behavior \eqref{hyp:exist_LM} and that the diffusion rates are such that
\begin{equation}
\label{hyp:di}
\inf\limits_{i\in\N^*} d_i>0 \quad \text{and}\quad \sup\limits_{i\in\N^*} d_i <+\infty.
\end{equation}
Assume also that the initial data $c_i^{in}\geq 0$ are such that there exists $p>2$ such that the initial mass $\rho_1^{in}$ lies in $L^p(\Omega)$. 
\par 
Then there exists a weak solution to the coagulation-fragmentation system \eqref{eq:syst_coag-frag}-\eqref{hyp:modele} such that $\rho_1$ lies in $L^p(\Omega_T)$ for all positive $T$.
\end{theorem}
\begin{remark}
We point out that the only assumption on the coagulation coefficients in Theorem~\ref{th:existence_gene} is~\eqref{hyp:exist_LM}, therefore this result includes situations where gelation can occur.
\end{remark}

\noindent Finally, it was shown in~\cite{BreDevFel16}, with sublinear coagulation and no fragmentation, that having a control on all moments $\rho_k$, $k\geq 1$, in all $L^p(\Omega_T)$, $p<\infty$, was enough to get a result of propagation of smoothness for solutions of~\eqref{eq:syst_coag-frag}-\eqref{hyp:modele}. Since we have such a control in the strong fragmentation case thanks to Theorem~\ref{th:strong_frag}, we are now able to complete this smoothness result so that it encompasses nearly all situations where it is known that gelation can not happen. 

\begin{theorem}\label{th:all_moments_w}
Let $\Omega$ be a smooth bounded domain of $\R^N$, $N\leq 2$, and $T>0$. Make the assumption \eqref{hyp:convergence_di} of convergence of the diffusion coefficients and assume the following behavior for the coagulation and fragmentation coefficients:
\begin{equation*}
a_{i,j} \leq C_Q\left(i^{\alpha}j^{\beta}+i^{\beta}j^{\alpha}\right) \quad \text{and}\quad B_i\leq C_{max}i^{\gamma_{max}},
\end{equation*}
where $0\leq\alpha,\beta\leq 1$. Assume also that the initial data $c_i^{in}\geq 0$ are of class $\C^{\infty}(\overline \Omega)$, compatible with the boundary conditions and that for all $k \in \N^*$ the initial moments $\rho_k^{in}$ are of class $\C^{\infty}(\overline \Omega)$. Finally, assume that we are in one of the two following cases.

\medskip
\noindent\textsc{Sublinear coagulation case:} \eqref{hyp:exist_LM} holds and $\alpha+\beta<1$.

\medskip
\noindent\textsc{Strong fragmentation case:} \eqref{hyp:exist_frag_forte} holds and there exists $C_F>0$ and $\gamma\geq 1$, $\gamma>\alpha+\beta$, such that $B_i\geq C_F i^{\gamma}$ for all $i\in\N^*.$

\medskip
Then there exists a smooth solution to the coagulation-fragmentation system \eqref{eq:syst_coag-frag}-\eqref{hyp:modele} such that each $c_i$ is of class $\C^{\infty}(\overline\Omega_T)$, and such that the moments $\rho_k$ are also of class $\C^{\infty}(\overline\Omega_T)$, for any $k\in\N^*$. Besides, if $\sup\limits_{i\in\N^*} B_i<\infty$, the solution is unique. 
\end{theorem}

Let us make a few comments about Theorem~\ref{th:all_moments_w}
\begin{itemize}
\item The $\C^{\infty}$ regularity down to time $0$ requires of course the $\C^{\infty}$ hypothesis on the initial data. However, it can be seen in the various steps of the proof (see Section~\ref{sec:smooth}) that propagation of regularity  in intermediate Sobolev spaces holds under suitable (less stringent) assumptions on the initial data.
\item Since each $c_i$ is solution of a heat equation subject to a r.h.s. that can be controlled once all moments are bounded in $L^p(\Omega_T)$, $p<+\infty$, we can in fact show the creation of regularity for strictly positive times. For example, under the assumption that $\rho_k^{in}\in L^p(\Omega)$ for all $p<+\infty$ and all $k\in\mathbb{N}^*$, we can prove that the concentrations $c_i$  are of class $C^{\infty}(]0,T] \times \bar{\Omega})$.
\item In the strong fragmentation case, we can suppose even less, since we saw in Theorem~\ref{th:strong_frag} that moments bounds in $L^p(\Omega_T)$ are created. Indeed, if we simply assume $\rho_{k_0}^{in}\in L^1(\Omega)$ for a $k_0>1$, \eqref{eq:moments_strong_frag} yields that $\rho_k\in L^1([\varepsilon,T]\times\Omega)$ for all $\varepsilon>0$ and all $k\geq 0$, and this is then enough to show smoothness on $]0,T]\times\Omega$.
\end{itemize}

\medskip
Our paper is organised as follow. In Section~\ref{sec:approx} we present the truncated system we use to appoximate the caogulation-fragmentation system~\eqref{eq:syst_coag-frag}-\eqref{hyp:modele}. We also give sufficient conditions, in terms of \emph{a priori} estimates on the mass $\rho_1$, under which the solutions of the truncated system converge (up to extraction) to weak solutions of the full system~\eqref{eq:syst_coag-frag}-\eqref{hyp:modele}. In Section~\ref{sec:rappels} we then recall some key results from \cite{CanDevFel14} and \cite{BreDevFel16}: improved duality lemmas that can be used to get the \textit{a priori} estimates introduced in the previous section, allowing us to prove Theorem~\ref{th:existence_gene}. In Section~\ref{sec:preuve_frag_forte} we make further use of these \emph{a priori} estimates, and show that they enable creation of higher order moments in the strong fragmentation case, to get the proof of Theorem~\ref{th:strong_frag}. We conclude in Section~\ref{sec:smooth} by showing how control on higher order moments translates into smoothness and prove Theorem~\ref{th:all_moments_w}.

\section{Approximation scheme and existence results}
\label{sec:approx}

In this section, we explain how to obtain weak solutions of the coagulation-fragmentation system~\eqref{eq:syst_coag-frag}-\eqref{hyp:modele} from a truncated system. For all $n\in\N^*$, we consider $c^n=(c_1^n,\ldots,c_n^n)$ the solution of
\begin{align}
\label{eq:approx_system}
\left\{
\begin{aligned}
& \partial_t c_i^n - d_i \Delta_x c_i^n = Q_i^n(c^n) + F_i^n(c^n), \quad &\text{on}& \ \Omega_T, \\
& \nabla_x c_i^n\cdot \nu = 0 \quad &\text{on}& \ [0,T]\times\partial\Omega, \\
& c_i^n(0,\cdot) = c_i^{in} \quad &\text{on}& \ \Omega,
\end{aligned}
\right.
\end{align}
with
\begin{equation}
\label{eq:approx_Q}
Q_i^n(c^n) = Q_i^{+,n}(c^n) - Q_i^{-,n}(c^n) = \frac{1}{2}\sum_{j=1}^{i-1}a_{i-j,j}c_{i-j}^nc_j^n-\sum_{j=1}^{n-i}a_{i,j}c_i^nc_j^n,
\end{equation}
and
\begin{equation}
\label{eq:approx_F}
F_i^n(c^n) = F_i^{+,n}(c^n) - F_i^{-,n}(c^n) = \sum_{j=1}^{n-i}B_{i+j}\beta_{i+j,i}c_{i+j}^n - B_ic_i^n.
\end{equation}
For this reaction diffusion system (of finite dimension and with finite sums in the r.h.s.), existence and uniqueness of nonnegative, global and smooth solution is known (see for instance \cite{Wrz97}). Notice that, if one assumes~\eqref{hyp:modele}, the truncated version of the weak formulation~\eqref{eq:form_faible_Q}-\eqref{eq:form_faible_F} becomes (for any sequence $\left(\varphi_i\right)$)
\begin{equation}
\label{eq:Q_faible_tronq}
\sum_{i=1}^{n} \varphi_iQ_i^n(c^n) = \frac{1}{2}\sum_{\substack{1\leq i,j\leq n \\ i+j\leq n }}a_{i,j}c_i^nc_j^n(\varphi_{i+j}-\varphi_i-\varphi_j),
\end{equation}
and
\begin{equation}
\label{eq:F_faible_tronq}
\sum_{i=1}^{n} \varphi_iF_i^n(c^n) = -\sum_{i=2}^{n}B_ic_i^n\left(\varphi_i-\sum_{j=1}^{i-1}\beta_{i,j}\varphi_j\right).
\end{equation}
For all $k\geq 0$, we then define the moment of order $k$, associated to the solution $c^n$, as 
\begin{equation*}
\rho_k^n=\sum_{i=1}^ni^kc_i^n.
\end{equation*}
We now given sufficient conditions on those moments, under which one can extract from $\left(c^n\right)_n$ a subsequence converging to a solution of~\eqref{eq:syst_coag-frag}-\eqref{hyp:modele}.
\begin{proposition}
\label{prop:existence}
Let $\Omega$ be a smooth bounded domain of $\R^N$. Assume that $d_i>0$ for all $i\in\N^*$, that there exists $p>2$ such that, for all $n\in\N^*$, the mass $\left(\rho_1^n\right)_n$ associated to the solution $c^n$ of~\eqref{eq:approx_system}-\eqref{eq:approx_F} is bounded in $L^p(\Omega_T)$. Assume also that one of the two following statements holds:
\begin{enumerate}[(i)]
\item the coagulation and fragmentation coefficients satisfy the asymptotic behavior given in~\eqref{hyp:exist_LM},
\item the coagulation and fragmentation coefficients satisfy the (weaker) asymptotic behavior given in~\eqref{hyp:exist_frag_forte}, and there exists $k>1$ such that $\left(\rho_k^n\right)_n$ is bounded in $L^1(\Omega_T)$.
\end{enumerate}
Then, up on extraction, $\left(c^n\right)_n$ converges to a weak solution $c$ of~\eqref{eq:syst_coag-frag}-\eqref{def:F}, for which $\rho_1$ lies in $L^p(\Omega_T)$.
\end{proposition}
\begin{remark}
\begin{itemize}
\item As mentioned in the introduction, for the case (i) a more general result was already proven in~\cite{LauMis02}. However, our proof is much simpler because we assume that we have an a priori estimate on the mass in $L^p$ for some $p>2$, from which we can deduce the convergence (up to extraction) of $\left(c_i^n\right)_n$ in $L^2(\Omega_T)$. We show situations where we can get this a priori estimate in the next sections. A proof similar to the one given here was also presented in~\cite{CanDevFel10}, but it was only valid in dimension $N=1$ because they only had an a priori estimate in $L^2$ for the mass.
\item While assumption~\eqref{hyp:exist_LM} covers a wide variety of coagulation and fragmentation coefficients, it does not hold in some of the cases of strong fragmentation (which allows for stronger coagulation) considered in Theorem~\ref{th:strong_frag}, for instance if $\alpha=1$ or $\beta=1$. In that case the existence will be provided by the case (ii) of Proposition~\ref{prop:existence}. 
\end{itemize}
\end{remark}
\begin{proof}
Since $p\geq 2$, we start by noticing that, for all $i\in\N^*$ $\left(Q_i^n(c^n)\right)_n$ and $\left(F_i^n(c^n)\right)_n$ are bounded, in $L^{1}(\Omega_T)$ and $L^2(\Omega_T)$ respectively. Indeed, remembering the notations introduced in~\eqref{hyp:exist_frag_forte},
\begin{equation*}
\left\vert Q_i^n(c^n) \right\vert \leq  \frac{1}{2}\sum_{j=1}^{i-1}a_{i-j,j}\left(\left(c_{i-j}^n\right)^2+\left(c_j^n\right)^2\right) + c_i^n\sum_{j=1}^{n-i}\frac{a_{i,j}}{j}jc_j^n \leq \left(\rho_1^n\right)^2\left(\sum_{j=1}^{i-1}a_{i-j,j}+K_i^Q\right)
\end{equation*}
and
\begin{equation*}
\left\vert F_i^n(c^n) \right\vert \leq \sum_{j=1}^{n-i}\frac{B_{i+j}\beta_{i+j,i}}{i+j}(i+j)c_{i+j}^n + B_ic_i^n \leq \rho_1^n \left(K_i^F+B_i\right).
\end{equation*}
Therefore $\left(\partial_t c_i^n - d_i\Delta_x c_i^n\right)_n$ is bounded in $L^1(\Omega_T)$, which yields that $\left(c_i^n\right)_n$ lies in a (strongly) compact subset of $L^1(\Omega_T)$. Thus, up to a diagonal extraction, we can assume that for each $i$ in $\N^*$,
\begin{equation*}
c_i^n \longrightarrow c_i,\quad \text{in }L^1(\Omega_T) \text{ and almost surely,} \quad \forall i\in\N^*.
\end{equation*}
Then, since $\left(c_i^n\right)_n$ is bounded in $L^p(\Omega_T)$, (by $\left(\rho_1^n\right)_n$), Fatou's Lemma yields that $c_i$ does in fact lie in $L^p(\Omega_T)$ and since $p>2$, we get by interpolation that 
\begin{equation*}
c_i^n \longrightarrow c_i,\quad \text{in }L^2(\Omega_T), \quad \forall i\in\N^*.
\end{equation*}
Still by Fatou's Lemma, we also have that $\rho_1$ lies in $L^p(\Omega_T)$.

To prove that the limit $c=(c_i)$ is a weak solution of \eqref{eq:syst_coag-frag}-\eqref{hyp:modele}, it suffices to show that $\left(Q_i^n(c^n)\right)_n$ and $\left(F_i^n(c^n)\right)_n$ converge in $L^1(\Omega_T)$ to $Q_i(c)$ and $F_i(c)$ respectively. The convergence of $\left(Q_i^{+,n}(c^n)\right)_n$ and $\left(F_i^{-,n}(c^n)\right)_n$ to $Q_i^{+}$ and $F_i^{-}$ respectively is obvious, because these terms are only composed of linear or quadratic finite sums (and we have the convergence of $c_i^n$ in $L^2$ so the quadratic terms do converge in $L^1$). For the remaining terms, we need a different argument for the cases (i) and (ii).
If (i) holds, we estimate
\begin{align*}
\left\vert Q_i^{-,n}(c^n) - Q_i^-(c)\right\vert &\leq \sum_{j=1}^{j_0}a_{i,j}\left\vert c_{i}^nc_j^{n} -c_{i}c_j \right\vert + \sup\limits_{j>j_0}\frac{a_{i,j}}{j}\left(c_i\sum_{j=j_0+1}^{\infty} jc_j + c_i^n\sum_{j=j_0+1}^{n} jc_j^n\right)\\
&\leq \sum_{j=1}^{j_0}a_{i,j}\left\vert c_{i}^nc_j^{n} -c_{i}c_j \right\vert + \sup\limits_{j>j_0}\frac{a_{i,j}}{j}\left(\rho_1^2+(\rho_1^n)^2\right),
\end{align*}
and 
\begin{align*}
\left\vert F_i^{+,n}(c^n) - F_i^+(c)\right\vert &\leq \sum_{j=1}^{j_0}B_{i+j}\beta_{i+j,i}\left\vert c_{i+j}^n -c_{i+j} \right\vert \\
&\quad + \sup\limits_{j>j_0} \frac{B_{i+j}\beta_{i+j,i}}{i+j}\left(\sum_{j=j_0+1}^{\infty} (i+j)c_{i+j} + \sum_{j=j_0+1}^{n} (i+j)c_{i+j}^n\right)\\
&\leq \sum_{j=1}^{j_0}B_{i+j}\beta_{i+j,i}\left\vert c_{i+j}^n -c_{i+j} \right\vert + \sup\limits_{j>j_0} \frac{B_{i+j}\beta_{i+j,i}}{i+j}\left(\rho_1+\rho_1^n\right).
\end{align*}
Combining the convergence of $\left(c_i^n\right)_n$ to $c_i$ in $L^2$ with~\eqref{hyp:exist_LM} and the fact that $\left(\rho_1^n\right)_n$ is bounded in $L^p(\Omega_T)$, $p\geq 2$, we get that $\left(Q_i^{-,n}(c^n)\right)_n$ and $\left(F_i^{+,n}(c^n)\right)_n$ converge in $L^1(\Omega_T)$, to $Q_i^{-}(c)$ and $F_i^{+}(c)$ respectively. This concludes the proof in the case (i).

If (ii) holds instead of (i), we can only get 
\begin{align*}
\left\vert Q_i^{-,n}(c^n) - Q_i^-(c)\right\vert &\leq \sum_{j=1}^{j_0}a_{i,j}\left\vert c_{i}^nc_j^{n} -c_{i}c_j \right\vert + K^Q_i\left(c_i\sum_{j=j_0+1}^{\infty} jc_j + c_i^n\sum_{j=j_0+1}^{n} jc_j^n\right),
\end{align*}
and 
\begin{align*}
\left\vert F_i^{+,n}(c^n) - F_i^+(c)\right\vert &\leq \sum_{j=1}^{j_0}B_{i+j}\beta_{i+j,i}\left\vert c_{i+j}^n -c_{i+j} \right\vert \\
&\quad + K^F_i\left(\sum_{j=j_0+1}^{\infty} (i+j)c_{i+j} + \sum_{j=j_0+1}^{n} (i+j)c_{i+j}^n\right).
\end{align*}
However, the knowledge that a superlinear moment (i.e. $\rho_k^n$ for $k>1$) is bounded in $L^1(\Omega_T)$, will allow us to control (uniformly in $n$) reminders terms of the form $\sum_{j=j_0+1}^{n} jc_j^n$. Indeed, we already know that $\sum_{j=j_0+1}^{\infty} jc_j$ can be made arbitrarily small in $L^2(\Omega_T)$ by taking $j_0$ large enough, since $\rho_1$ lies in $L^p(\Omega_T)$ with $p\geq 2$. To show that $\sum_{j=j_0+1}^{n} jc_j^n$ can also be made arbitrarily small in $L^2(\Omega_T)$ (uniformly in $n$), it is sufficient to prove that $\left(\rho_1^n\right)_n$ converges to $\rho_1$ in $L^2(\Omega_T)$. To do so, we estimate
\begin{align}
\label{eq:rho_1_convergence}
\int\limits_{\Omega_T}\left\vert \rho_1-\rho_1^n\right\vert  &\leq \sum_{i=1}^{i_0-1}i\int\limits_{\Omega_T}\left\vert c_i-c_i^n\right\vert + \int\limits_{\Omega_T}\sum_{i=i_0}^{\infty}ic_i + \int\limits_{\Omega_T}\sum_{i=i_0}^{n}ic_i^n \nonumber\\
&\leq \sum_{i=1}^{i_0-1}i\int\limits_{\Omega_T}\left\vert c_i-c_i^n\right\vert + \frac{1}{i_0^{k-1}}\left(\int\limits_{\Omega_T}\sum_{i=i_0}^{\infty}i^kc_i + \int\limits_{\Omega_T}\sum_{i=i_0}^{n}i^kc_i^n\right) \nonumber\\
&\leq \sum_{i=1}^{i_0-1}i\int\limits_{\Omega_T}\left\vert c_i-c_i^n\right\vert + \frac{1}{i_0^{k-1}}\left(\left\Vert \rho_k\right\Vert_{L^1(\Omega_T)} + \left\Vert \rho_k^n\right\Vert_{L^1(\Omega_T)} \right),
\end{align}
where we know that $\rho_k$ lies in $L^1(\Omega_T)$, again by Fatou's Lemma. Therefore, $\left(\rho_1^n\right)_n$ converges to $\rho_1$ in $L^1(\Omega_T)$, but since $\left(\rho_1^n\right)_n$ is bounded in $L^p(\Omega_T)$, $p>2$, we get by interpolation that $\left(\rho_1^n\right)_n$ also converges to $\rho_1$ in $L^2(\Omega_T)$. Thus, we finally get that
\begin{equation*}
\left(c_i\sum_{j=j_0+1}^{\infty} jc_j + c_i^n\sum_{j=j_0+1}^{n} jc_j^n\right)\quad \text{and}\quad\left(\sum_{j=j_0+1}^{\infty} (i+j)c_{i+j} + \sum_{j=j_0+1}^{n} (i+j)c_{i+j}^n\right)
\end{equation*}
can be made arbitrarily small in $L^1(\Omega_T)$ (uniformly in $n$) by taking $j_0$ large enough, which shows that $\left(Q_i^{-,n}(c^n)\right)_n$ and $\left(F_i^{+,n}(c^n)\right)_n$ converge in $L^1(\Omega_T)$, to $Q_i^{-}(c)$ and $F_i^{+}(c)$ respectively. This concludes the proof in the case (ii).
\end{proof}
\begin{remark}
We point out that for case (ii), the fragmentation part of hypothesis~\eqref{hyp:exist_frag_forte} is not optimal. Indeed when proving the convergence of $F_i^n$ to $F_i$, we only used information on the first moment (more precisely that $\rho_1$ lies in $L^1(\Omega_T)$ and that $\rho_1^n$ converges to $\rho_1$ in $L^1(\Omega_T)$). Knowing that $\left(\rho_k^n\right)_n$ is bounded in $L^1(\Omega_T)$, for some $k>1$, we could show the convergence of $F_i^n$ to $F_i$ in $L^1(\Omega_T)$ with a weaker hypothesis than~\eqref{hyp:exist_frag_forte}. Indeed estimating
\begin{align*}
\left\vert F_i^{+,n}(c^n) - F_i^+(c)\right\vert \leq &\ \sum_{j=1}^{j_0}B_{i+j}\beta_{i+j,i}\left\vert c_{i+j}^n -c_{i+j} \right\vert \\
&\ + \sup\limits_{j> j_0} \frac{B_{i+j}\beta_{i+j,i}}{(i+j)^k}\left(\sum_{j=j_0+1}^{\infty} (i+j)^kc_{i+j} + \sum_{j=j_0+1}^{n} (i+j)^kc_{i+j}^n\right),
\end{align*}
we see that is suffices to assume that
\begin{equation*}
\lim\limits_{j\to\infty} \frac{B_{i+j}\beta_{i+j,i}}{(i+j)^k}=0, \quad \forall~i\in\N^*.
\end{equation*}
Here it seems mandatory to assume that the limit is 0 (and not simply that the supremum  over $j$ is finite) because we do not know that $\left(\rho_k^n\right)_n$ converges to $\rho_k$ but only that is is bounded in $L^1(\Omega_T)$. Also, we cannot so easily extend similarly the hypothesis on the coagulation coefficients because of the quadratic nature of $Q_i$.
\end{remark}

\section{Duality estimates and propagation of the mass in $L^p$ norms}
\label{sec:rappels}

In this section, we present some duality lemmas and their applications to the proof of propagation of mass in $L^p$ norm for the coagulation-fragmentation system~\eqref{eq:syst_coag-frag}-\eqref{hyp:modele}. First we need to introduce the

\begin{definition}
\label{def:Cmp}
For $m>0$ and $q\in]1,+\infty[$, we define $\constC_{m,q} > 0$ as the best (i.e. the smallest) constant independent of $T>0$ in the parabolic maximal regularity estimate 
\begin{equation*}
\left(\,\int_{\Omega_T} \left\vert \partial_t v \right\vert^q +m^q \int_{\Omega_T} \left\vert \Delta_x v \right\vert^q\right)^{\frac{1}{q}} \leq \constC_{m,q} \left(\,\int_{\Omega_T} \left\vert f \right\vert^q\right)^{\frac{1}{q}},
\end{equation*}
where $v$ is the unique solution of the heat equation with constant diffusion coefficient $m$, homogeneous Neumann boundary conditions, and zero initial data:
\begin{equation*}
\left\{
\begin{aligned}
& \partial_t v - m\Delta_x v = f \quad &\text{on}& \ \Omega_T, \\
& \nabla_x v\cdot \nu = 0 \quad &\text{on}& \ [0,T]\times\partial\Omega, \\
&  v(0,\cdot) = 0 \quad &\text{on}& \ \Omega.
\end{aligned}
\right.
\end{equation*}
\end{definition}
\noindent We recall that for all $m>0$ and $q\in]1,+\infty[$, $\constC_{m,q}$ is finite, and for the particular case $q=2$ we have an explicit bound $\constC_{m,2}\leq 1$ (see \cite{BreDevFel16} and the references therein). We now recall some duality results, the first one being Proposition~1.1 of \cite{CanDevFel14} (see also Proposition~2.4 of \cite{BreDevFel16} for this exact formulation) and the second one being Proposition~2.5 of \cite{BreDevFel16}. In the sequel, we shall systematically denote by $p'$ the conjugate exponent of $p\in]1,+\infty[$, e.g. satisfying $\frac{1}{p}+\frac{1}{p'}=1$.
\begin{lemma}\label{prop:duality}
Let $\Omega$ be a smooth bounded subset of $\R^N$ and consider a function $M:\Omega_T\to\R_+$ satisfying $a\leq M \leq b$ for some $a,b>0$. For any $p\in]1,+\infty[$, if 
\begin{equation*}
\frac{b-a}{b+a}\constC_{\frac{a+b}{2},p'}<1,
\end{equation*}
then there exists a constant $C$ (depending only on $\Omega,T,a,b,p$) such that for any $u_0\in L^p(\Omega)$, any weak solution $u$ of the parabolic system
\begin{equation*}
\left\{
\begin{aligned}
& \partial_t u - \Delta_x \left(Mu\right) = 0 \quad &\text{on}& \ \Omega_T, \\
& \nabla_x u\cdot \nu = 0 \quad &\text{on}& \ [0,T]\times\partial\Omega, \\
&  u(0,\cdot) = u_0 \quad &\text{on}& \ \Omega,
\end{aligned}
\right.
\end{equation*}
satisfies  $\left\Vert u\right\Vert_{L^p(\Omega_T)} \leq C \left\Vert u_0\right\Vert_{L^p(\Omega)}$.
\end{lemma}

\begin{lemma}\label{prop:duality_rhs}
Let $\Omega$ be a smooth bounded subset of $\R^N$, $\mu_1,\mu_2\geq 0$ and consider a function $M:\Omega_T\to\R_+$ satisfying $a\leq M \leq b$ for some $a,b>0$. For any $p\in]1,+\infty[$, if 
\begin{equation*}
\frac{b-a}{b+a}\constC_{\frac{a+b}{2},p'}<1,
\end{equation*}
then there exists a constant $C$ (depending only on $\Omega,T,a,b,p$) such that for any $u_0\in L^p(\Omega)$, any function $u:\Omega_T\to \R_+$ satisfying (weakly)
\begin{equation}
\label{eq:duality_2}
\left\{
\begin{aligned}
& \partial_t u - \Delta_x\left(Mu\right) \leq \mu_1 u + \mu_2 \quad &\text{on}& \ \Omega_T,\\
& \nabla_x u\cdot \nu = 0 \quad &\text{on}& \ [0,T]\times\partial\Omega, \\
& u(0,\cdot) = u_0 \quad &\text{on}& \ \Omega,
\end{aligned}
\right.
\end{equation}
belongs  to $L^p(\Omega_T)$, and satisfies the estimate:
\begin{equation*}
\left\Vert u\right\Vert_{L^p(\Omega_T)} < C \left(1+\left\Vert u_0\right\Vert_{L^p(\Omega)}\right).
\end{equation*}
\end{lemma}

Let us now briefly explain how the above duality lemmas are used in the context of the coagulation-fragmentation system~\eqref{eq:syst_coag-frag}-\eqref{hyp:modele}.
\begin{proposition}\label{th:first_moment}
Let $\Omega$ be a smooth bounded domain of $\R^N$ and $T>0$. Assume~\eqref{hyp:modele}. Let $p\in]1,+\infty[$ and assume also that the initial data $c_i^{in}\geq 0$ are such that the initial mass $\rho_1^{in}$ lies in $L^p(\Omega)$. Assume
\begin{equation}
\label{def_hyp:ab}
a:=\inf\limits_{i\in\N^*} d_i>0 \quad \text{and} \quad b:=\sup\limits_{i\in\N^*} d_i<\infty,
\end{equation}
and
\begin{equation}
\label{hyp:closeness}
\frac{b-a}{b+a}\constC_{\frac{a+b}{2},p'}<1.
\end{equation}
\par 
Then there exists $C>0$ such that, for all $n\in\N^*$, the mass $\rho_1^n$ associated to  the solution of the truncated system \eqref{eq:approx_system}-\eqref{eq:approx_F} satisfies
\begin{equation*}
\left\Vert \rho_1^n\right\Vert_{L^p(\Omega_T)}\leq C \left\Vert \rho_1^ {in}\right\Vert_{L^p(\Omega)}.
\end{equation*}
\end{proposition}
\begin{proof}
Looking at~\eqref{eq:Q_faible_tronq}-\eqref{eq:F_faible_tronq}, we see that
\begin{equation*}
\sum_{i=1}^niQ_i^n = 0 = \sum_{i=1}^niF_i^n,
\end{equation*}
and thus the mass $\rho_1^n$ satisfies
\begin{equation*}
\partial_t \rho_1^n - \Delta_x \left(M_1^n\rho_1^n\right) = 0,
\end{equation*}
where $M_1^n=\frac{\sum_{i=1}^{n} i d_ic_i^n}{\sum_{i=1}^{n} i c_i^n}$. Thanks to assumptions~\eqref{def_hyp:ab}-\eqref{hyp:closeness}, we can apply Lemma~\ref{prop:duality} to $\rho^n_1$, and the results follows since $\Vert \left(\rho_1^n\right)^{in}\Vert_{L^p(\Omega)} \leq \Vert \rho_1^ {in}\Vert_{L^p(\Omega)}$ for all $n\in\N^*$.
\end{proof}

As already pointed out in~\cite{CanDevFel14}, the closeness hypothesis~\eqref{hyp:closeness} can be removed when $p$ is close to 2, providing a variant of Proposition~\ref{th:first_moment}. 
\begin{proposition}\label{cor:proche_2}
Let $\Omega$ be a smooth bounded domain of $\R^N$ and $T>0$. Assume~\eqref{hyp:modele}. Assume that the diffusion coefficients satisfy \eqref{def_hyp:ab} and that we have nonnegative initial data $c_i^{in}\geq 0$.
\par
There exists $p_0>2$ such that, for all $p\in [2,p_0[$,  if the initial mass $\rho_1^{in}$ lies in $L^p(\Omega)$, then there exists $C>0$ such that, for all $n\in\N^*$, the mass $\rho_1^n$ associated to  the solution of the truncated system \eqref{eq:approx_system}-\eqref{eq:approx_F} satisfies
\begin{equation*}
\left\Vert \rho_1^n\right\Vert_{L^p(\Omega_T)}\leq C \left\Vert \rho_1^ {in}\right\Vert_{L^p(\Omega)}.
\end{equation*}
\end{proposition}
\begin{proof}
As already pointed out after Definition~\ref{def:Cmp}, for $q=2$ we have an explicit bound $\constC_{m,2}\leq 1$. This yields that for any $a,b>0$,
\begin{equation*}
\frac{b-a}{b+a}\constC_{\frac{a+b}{2},2}\leq \frac{b-a}{b+a}<1.
\end{equation*}
Therefore, by continuity of $p\mapsto \constC_{m,p'}$ (see \cite[Lemma~3.2]{CanDevFel14} for an explicit computation) there exists $p_0>2$ (depending on $a=\inf\limits_{i\in\N^*} d_i$ and $b=\sup\limits_{i\in\N^*} d_i$), such that for all $p\in [2,p_0[$
\begin{equation*}
\frac{b-a}{b+a}\constC_{\frac{a+b}{2},p'}<1.
\end{equation*}
We then conclude by applying again Lemma~\ref{prop:duality} (or directly Proposition~\ref{th:first_moment}).
\end{proof}

Finally, if one assume \eqref{hyp:convergence_di}, e.g. that the diffusion rates $d_i$ are converging toward a positive limit, hypothesis~\eqref{hyp:closeness} can be removed for any $p$, basically because the coefficients $d_i$ will then be arbitrarily close from one another for $i$ large enough (if fragmentation is non zero, we have to restrict ourselves to low dimension $N\leq 2$). This leads to yet another variant of Proposition~\ref{th:first_moment}, which was already stated and proven in \cite{BreDevFel16} in the particular case of no fragmentation, that is when one assumes that $F_i=0$ for all $i\in\N^*$.

\begin{proposition}\label{th:first_moment_convergence}
Let $\Omega$ be a smooth bounded domain of $\R^N$, $N\leq 2$, and $T>0$. Assume~\eqref{hyp:modele}. Assume also that the coagulation coefficients satisfy
\begin{equation}
\label{hyp:Cij}
a_{i,j}\leq Cij,\qquad \forall~i,j\in\N^*,
\end{equation} 
and that the diffusion coefficients satisfy~\eqref{hyp:convergence_di}. Let $p\in]2,+\infty[$ and assume that the initial data $c_i^{in}\geq 0$ lie in $L^{\infty}(\Omega)$ for all $i\in\N^*$, and that the initial mass $\rho_1^{in}$ lies in $L^p(\Omega)$. 
\par
Then there exists $C>0$ such that, for all $n\in\N^*$, the mass $\rho_1^n$ associated to  the solution of the truncated system \eqref{eq:approx_system}-\eqref{eq:approx_F} satisfies
\begin{equation*}
\label{eq:apriori_rho1}
\left\Vert \rho_1^n\right\Vert_{L^p(\Omega_T)}\leq C \left(1+\left\Vert \rho_1^ {in}\right\Vert_{L^p(\Omega)}\right).
\end{equation*}
\end{proposition}
We point out that the assumption~\eqref{hyp:convergence_di} stating that $(d_i)$ converges, which allows us to get rid of the closeness hypothesis~\eqref{hyp:closeness}, is not much more stringent from the physical point of view than just assuming that the diffusions coefficients are bounded below, because it is expected that the clusters diffuse less when they become larger, and thus the sequence $(d_i)$ is expected to be decreasing. Also notice that while assumption~\eqref{hyp:closeness} depends on $p$ (and gets more and more stringent when $p$ goes to infinity), assuming~\eqref{hyp:convergence_di} allows to get the propagation of the mass in every $L^p$, $p<+\infty$. To prove Proposition~\ref{th:first_moment_convergence}, we first need a control on a finite number of concentrations $c_i$. This is the content of the following lemma.
\begin{lemma}\label{lem:c_i}
Let $\Omega$ be a smooth bounded domain of $\R^N$, $N\leq 2$, and $T>0$. Assume~\eqref{hyp:modele}. Assume that the fragmentation coefficients satisfy the asymptotic behavior \eqref{hyp:exist_frag_forte} and that the diffusion coefficients satisfy \eqref{def_hyp:ab}. Assume also that the initial concentrations $c^{in}_i\geq 0$ each lie in $L^{\infty}(\Omega)$ and that there exists $p>2$ such that the initial mass $\rho_1^{in}$ lies in $L^p(\Omega)$. 
\par
Then, for all $i\in\N^*$, there exists $\tilde c_i\geq 0$ such that the solution $c^n$ of~\eqref{eq:approx_system}-\eqref{eq:approx_F} satisfies
\begin{equation*}
\left\Vert c^n_i \right\Vert_{L^{\infty}(\Omega_T)} \leq \tilde c_i,\quad \forall~n\in\N^*.
\end{equation*}
\end{lemma}
\begin{proof}
By Proposition~\ref{cor:proche_2}, $\left(\rho_1^n\right)_n$ is bounded in $L^r(\Omega_T)$, for $r=\min(p,p_0)>2$. From this we deduce that $\left(F_i^{+,n}(c^n)\right)_n$ is bounded in $L^r(\Omega_T)$. Indeed, remembering~\eqref{hyp:exist_frag_forte},
\begin{equation*}
F_i^{+,n}(c^n) = \sum_{j=1}^{n-i}\frac{B_{i+j}\beta_{i+j,i}}{i+j}(i+j)c_{i+j}^n \leq K_i^F\rho_1^n.
\end{equation*}
Noticing that $F_i^{-,n}(c^n),Q_i^{-,n}(c^n)\leq 0$ for all $i\in\N^*$, we get
\begin{equation*}
\partial_t c_i^n - d_i \Delta_x c_i^n \leq Q_i^{+,n}(c^n) + F_i^{+,n}(c^n).
\end{equation*}
For $i=1$ we have that $Q_i^{+,n}(c^n)=0$, and therefore 
\begin{equation*}
\partial_t c_1^n - d_1 \Delta_x c_1^n \leq F_1^{+,n}(c^n)\in L^{r}(\Omega_T).
\end{equation*}
Since $r>2$, $N\leq 2$ and $c_1\geq 0$, the regularizing properties of the heat equation yield that $\left(c_1^n\right)_n$ is bounded in $L^{\infty}(\Omega_T)$.
Then for any integer $i\geq 2$, since 
\begin{align*}
\left\Vert Q_i^{+,n}(c^n)\right\Vert_{L^r(\Omega_T)} &\leq \left(\vert\Omega\vert T\right)^{\frac{1}{r}}\left\Vert Q_i^{+,n}(c^n)\right\Vert_{L^{\infty}(\Omega_T)}\\
& \leq \frac{\left(\vert\Omega\vert T\right)^{\frac{1}{r}}}{2}\sum_{j=1}^{i-1}a_{i,j}\left\Vert c_{i-j}^n\right\Vert_{L^{\infty}(\Omega_T)}\left\Vert c_j^n\right\Vert_{L^{\infty}(\Omega_T)}
\end{align*}
involves only $c_j^n$ for $j<i$, we can conclude by induction.
\end{proof}

We can now get $L^p$ estimates on the mass $\rho_1$, assuming the convergence~\eqref{hyp:convergence_di} of the diffusion coefficients.
\newline  \textit{Proof of Proposition~\ref{th:first_moment_convergence}.}
For any $I\in\N^*$, we define
\begin{equation*}
a^I:=\inf\limits_{i\geq I} d_i \quad \text{and} \quad b^I:=\sup\limits_{i\geq I} d_i.
\end{equation*}
Since $(d_i)$ converges toward a positive limit, there exists a positive integer $I$ such that
\begin{equation}
\label{eq:hyp_I}
\frac{b^I-a^I}{b^I+a^I}\constC_{\frac{a^I+b^I}{2},p'}<1.
\end{equation}
We fix such an $I$ and then consider, for $n\geq I$,
\begin{equation*}
\rho_1^{I,n}:=\sum_{i=I}^{n} ic_i^n \quad\text{and}\quad  M^{I,n}_1:=\frac{\sum_{i=I}^{n} i d_ic_i^n}{\sum_{i=I}^{n} i c_i^n}.
\end{equation*}
Since~\eqref{hyp:convergence_di} implies~\eqref{def_hyp:ab}, Lemma~\ref{lem:c_i} holds and it is enough to prove that $(\rho_1^{I,n})_n$ is bounded in $L^p(\Omega_T)$ to get that the full first moment $\left(\rho_1^n\right)_n$ is bounded in $L^p(\Omega_T)$. 

Using~\eqref{eq:F_faible_tronq} and the hypothesis $i=\sum_{j=1}^{i-1}j\beta_{i,j}$ of \eqref{hyp:modele}, we get that the contribution of fragmentation to the evolution of $\rho_1^{I,n}$ is nonpositive:
\begin{equation*}
\sum_{i=I}^{n}iF_i^n(c^n) = -\sum_{i=I}^{n}B_ic_i\left(i-\sum_{j=I}^{i-1}j\beta_{i,j}\right) \leq 0,
\end{equation*}
Therefore we get (using~\eqref{eq:Q_faible_tronq} and again the symmetry assumption in~\eqref{hyp:modele})
\begin{align*}
\partial_t \rho_1^{I,n} - \Delta_x \left(M_1^{I,n}\rho_1^{I,n}\right) &\leq \frac{1}{2} \sum_{i=1}^{n}\sum_{j=1}^{n}a_{i,j}c_i^nc_j^n\left((i+j)\mathds{1}_{i+j\geq I}-i\mathds{1}_{i\geq I}-j\mathds{1}_{j\geq I}\right)\\
&\leq \frac{1}{2} \sum_{i=1}^{n}\sum_{j=1}^{n}a_{i,j}c_i^nc_j^n\left(i\left(\mathds{1}_{i+j\geq I}-\mathds{1}_{i\geq I}\right)+j\left(\mathds{1}_{i+j\geq I}-\mathds{1}_{j\geq I}\right)\right)\\
&= \sum_{i=1}^{n}\sum_{j=1}^{n}a_{i,j}c_i^nc_j^ni\left(\mathds{1}_{i+j\geq I}-\mathds{1}_{i\geq I}\right).
\end{align*}
Using~\eqref{hyp:Cij}, we end up with
\begin{align*}
\partial_t \rho_1^{I,n} - \Delta_x \left(M_1^{I,n}\rho_1^{I,n}\right) &\leq C\sum_{i=1}^{I-1}i^2 c_i^n\sum_{j=1}^{n}jc_j^n \\
&= \psi_1^n\rho_1^{I,n} + \psi_2^n,
\end{align*}
where
\begin{equation*}
\psi_1^n= C\sum_{i=1}^{I-1}i^2 c_i^n \quad \text{and}\quad \psi_2^n=\psi_1\sum_{i=1}^{I-1}ic_i^n.
\end{equation*}
But $\left(\psi_1^n\right)_n$ and $\left(\psi_2^n\right)_n$ are bounded in $L^{\infty}(\Omega_T)$ by Lemma~\ref{lem:c_i}. Considering, for $k\in\{1,2\}$, $\mu_k$ a bound of $\left\Vert \psi_k^n \right\Vert_{L^{\infty}(\Omega_T)}$, we get
\begin{equation*}
\partial_t \rho_1^{I,n} - \Delta_x \left(M_1^{I,n}\rho_1^{I,n}\right) \leq  \mu_1 \rho_1^{I,n} + \mu_2,
\end{equation*}
and we can conclude the proof of Proposition~\ref{th:first_moment_convergence} thanks to Lemma~\ref{prop:duality_rhs}. \hfill $\qed$

\begin{remark}
Propositions~\ref{th:first_moment}, \ref{cor:proche_2} and \ref{th:first_moment_convergence} provide the kind of estimates on the mass $\rho_1^n$ that are needed to apply the existence result of Proposition~\ref{prop:existence} for the full system~\eqref{eq:syst_coag-frag}-\eqref{def:F}. In particular, Theorem~\ref{th:existence_gene} is nothing more than the combination of Proposition~\ref{cor:proche_2} and Proposition~\ref{prop:existence} case (i).
\end{remark}

\section{Superlinear moments and absence of gelation in the case of strong fragmentation}
\label{sec:preuve_frag_forte}

In this section we show how the result of Propositions~\ref{th:first_moment} and \ref{th:first_moment_convergence}, namely knowing that the mass $\rho_1$ lies in $L^p(\Omega_T)$, $p<+\infty$, allow us to prove Theorem~\ref{th:strong_frag}, that is the  creation and propagation of superlinear moments in presence of strong fragmentation, which prevents gelation. We start by presenting two elementary bounds that will be useful in the course of the proof.
\begin{lemma}
\label{lem:elem1}
Let $\xi\geq 0$, $C>0$ and $0<\theta<1$. Then
\begin{equation*}
\xi\leq C+\xi^{1-\theta} \quad\Rightarrow\quad \xi\leq\max\left(1,(1+C)^{\frac{1}{\theta}}\right).
\end{equation*}
\end{lemma}
\begin{lemma}
\label{lem:elem2}
Let $f$ be a nonnegative integrable function on $\R_+$, $m\in\N^*$ and $T>0$. Assume there exists $C_1,C_2>0$ and $0<\theta<1$ such that
\begin{equation*}
\int_0^T\frac{t^m}{m!}f(t)dt\leq C_1\int_0^T\frac{t^m}{m!}\left(f(t)\right)^{1-\theta}dt+C_2.
\end{equation*}
Then
\begin{equation*}
\int_0^T\frac{t^m}{m!}f(t)dt \leq 2\left(C_2 + \left(2C_1\right)^{\frac{1}{\theta}}\frac{T^{m+1}}{(m+1)!}\right).
\end{equation*}
\end{lemma}
\begin{proof}
Combining
\begin{align*}
\frac{1}{2}\int_0^T\frac{t^m}{m!}f(t)\mathds{1}_{\left(f(t)\right)^{\theta}\geq 2C_1}dt &\leq \int_0^T\frac{t^m}{m!}f(t)\left(1-\frac{C_1}{\left(f(t)\right)^{\theta}}\right)\mathds{1}_{\left(f(t)\right)^{\theta}\geq 2C_1}dt\\
&\leq C_2 - \int_0^T\frac{t^m}{m!}f(t)\left(1-\frac{C_1}{\left(f(t)\right)^{\theta}}\right)\mathds{1}_{\left(f(t)\right)^{\theta}< 2C_1}dt\\
&\leq C_2 + C_1\int_0^T\frac{t^m}{m!}\left(f(t)\right)^{1-\theta}\mathds{1}_{\left(f(t)\right)^{\theta}< 2C_1}dt\\
&\leq C_2 + \frac{\left(2C_1\right)^{\frac{1}{\theta}}}{2}\frac{T^{m+1}}{(m+1)!},
\end{align*}
with
\begin{equation*}
\int_0^T\frac{t^m}{m!}f(t)\mathds{1}_{\left(f(t)\right)^{\theta}< 2C_1}dt \leq \left(2C_1\right)^{\frac{1}{\theta}}\frac{T^{m+1}}{(m+1)!}
\end{equation*}
we get the announced estimate.
\end{proof}

We are now ready to prove Theorem~\ref{th:strong_frag}. The outline of the proof is the following. First, we use the assumption~\eqref{hyp:coef_frag_forte} to get a bound for any superlinear moment $\rho_l^n$ in terms of other moments, more precisely in terms of $\rho_{l+\gamma-1}^n$ and lower order moments. Then we use interpolation estimates to bound all the lower order moment in terms of $\rho_1^n$ and $\rho_{l+\gamma-1}^n$. But $\left(\rho_1^n\right)_n$ can be bounded by Proposition~\ref{th:first_moment_convergence}, and from this we deduce a bound on  $\rho_{l+\gamma-1}^n$ that does not depend on $n$. We then bootstrap the estimate to get a bound for $\rho_{l+m(\gamma-1)}^n$ for all $m\in\N^*$. Finally, these \emph{a priori} estimates enable us to apply Proposition~\ref{prop:existence}, and to get a bound on the moments $\rho_{l+m(\gamma-1)}$ of the full system.

\medskip

\noindent \textit{Proof of Theorem~\ref{th:strong_frag}.}
For $n\in\N^*$, we consider $c^n$ the solution of~\eqref{eq:approx_system}-\eqref{eq:approx_F}. 
For $l>1$ we introduce
\begin{equation*}
M_l^n=\frac{\sum_{i=1}^{n} i^l d_ic_i^n}{\sum_{i=1}^{n} i^l c_i^n}.
\end{equation*}
Using the weak formulation~\eqref{eq:Q_faible_tronq}-\eqref{eq:F_faible_tronq}, assumption~\eqref{hyp:coef_frag_forte} and the symmetry of~\eqref{hyp:modele}, we have
\begin{align}
\label{eq:rho_l_1}
&\partial_t \rho_l^n - \Delta_x \left(M_l^n\rho_l^n\right) \nonumber\\
&\qquad\qquad = \frac{1}{2}\sum_{\substack{1\leq i,j\leq n \\ i+j\leq n }}a_{i,j}c_i^nc_j^n\left((i+j)^l-i^l-j^l\right) - \sum_{i=2}^{n}B_ic_i^n\left(i^l-\sum_{j=1}^{i-1}j^l\beta_{i,j}\right)\nonumber \\
&\qquad\qquad \leq C_Q\sum_{i=1}^{n}\sum_{j=1}^{n}i^{\alpha}j^{\beta} \left((i+j)^l-i^l-j^l\right)c_i^nc_j^n - C_F\sum_{i=2}^{n}i^{\gamma}c_i^n\left(i^l-\sum_{j=1}^{i-1}j^l\beta_{i,j}\right).
\end{align}
To bound from above the coagulation term, we use the existence of a constant $C_{Q,l}>0$ such that, 
\begin{equation*}
(i+j)^l-i^l-j^l \leq C_{Q,l}(i^{l-1}j+ij^{l-1}),\quad \forall~i,j\in\N^*,
\end{equation*}
see for instance~\cite{Car92}. To bound from below the fragmentation term, we estimate (using~\eqref{hyp:modele})
\begin{align}
\label{eq:borne_beta}
i^l-\sum_{j=1}^{i-1}j^l\beta_{i,j} &= i^l\left(1-\frac{1}{i}\sum_{j=1}^{i-1}\left(\frac{j}{i}\right)^{l-1}j\beta_{i,j}\right) \nonumber\\
&\geq i^l\left(1-\left(\frac{i-1}{i}\right)^{l-1}\frac{1}{i}\sum_{j=1}^{i-1}j\beta_{i,j}\right) \nonumber\\
&= i^l\left(1-\left(1-\frac{1}{i}\right)^{l-1}\right) \nonumber\\
&\geq \min(l-1,1) i^{l-1},
\end{align}
the last inequality coming from the concavity (if $1<l\leq 2$) or the convexity (if $l\geq 2$) of $x\mapsto(1-x)^{l-1}$. Introducing $C_{F,l}=\min(l-1,1)>0$ and going back to~\eqref{eq:rho_l_1}, we end up with
\begin{align*}
&\partial_t \rho_l^n - \Delta_x \left(M_l^n\rho_l^n\right) \leq C_QC_{Q,l}\left(\rho_{\alpha+l-1}^n\rho_{\beta+1}^n+\rho_{\alpha+1}^n\rho_{\beta+l-1}^n\right)-C_FC_{F,l}\left(\rho_{\gamma+l-1}^n-c_1^n\right).
\end{align*}
Integrating on $\Omega$ and using the homogeneous Neumann boundary conditions, we get that
\begin{equation}
\label{eq:rho_l_2}
\frac{d}{dt}\int\limits_{\Omega}\rho_l^n + C_FC_{F,l}\int\limits_{\Omega}\rho_{\gamma+l-1}^n \leq C_QC_{Q,l}\int\limits_{\Omega}\left(\rho_{\alpha+l-1}^n\rho_{\beta+1}^n+\rho_{\alpha+1}^n\rho_{\beta+l-1}^n\right)+ C_FC_{F,l}\int\limits_{\Omega}c_1^n.
\end{equation}
Assuming that $l>2-(\gamma-\alpha)$ and $l>2-(\gamma-\beta)$, i.e. $\alpha+1\leq \gamma+l-1$ and $\beta+1\leq \gamma+l-1$ (these are the assumptions on $k$ in Theorem~\ref{th:strong_frag}), H\"older's inequality yields the following interpolation estimates
\begin{equation*}
\rho_{\alpha+1}^n\leq\left(\rho_{1}^n\right)^{\frac{\gamma+l-\alpha-2}{\gamma+l-2}}\left(\rho_{\gamma+l-1}^n\right)^{\frac{\alpha}{\gamma+l-2}},\qquad \rho_{\beta+1}^n\leq\left(\rho_{1}^n\right)^{\frac{\gamma+l-\beta-2}{\gamma+l-2}}\left(\rho_{\gamma+l-1}^n\right)^{\frac{\beta}{\gamma+l-2}},
\end{equation*}
and
\begin{equation*}
\rho_{\alpha+l-1}^n\leq\left(\rho_{1}^n\right)^{\frac{\gamma-\alpha}{\gamma+l-2}}\left(\rho_{\gamma+l-1}^n\right)^{\frac{\alpha+l-2}{\gamma+l-2}}, \qquad \rho_{\beta+l-1}^n\leq\left(\rho_{1}^n\right)^{\frac{\gamma-\beta}{\gamma+l-2}}\left(\rho_{\gamma+l-1}^n\right)^{\frac{\beta+l-2}{\gamma+l-2}},
\end{equation*}
Notice that in the case where $\alpha+l-1<1$ or $\beta+l-1<1$, the last two interpolation estimates are no longer valid, but we can then simply use $\rho_{\alpha+l-1}^n\leq\rho_{1}^n$ and $\rho_{\beta+l-1}^n\leq\rho_{1}^n$. The rest of the proof is then identical, up to different exponents for $\rho_1^n$ and $\rho_{\gamma+l-1}^n$. The only property that we are going to use, which holds in every case, is that the obtained exponent for $\rho_{\gamma+l-1}^n$ is strictly less than 1.

In the case $\alpha+l-1\geq 1$ and $\beta+l-1\geq 1$, \eqref{eq:rho_l_2} then becomes
\begin{align*}
&\frac{d}{dt}\int\limits_{\Omega}\rho_l^n + C_FC_{F,l}\int\limits_{\Omega}\rho_{\gamma+l-1}^n \\
&\qquad\qquad\qquad\qquad\qquad \leq 2C_QC_{Q,l}\int\limits_{\Omega}\left(\rho_{1}^n\right)^{\frac{2\gamma+l-(\alpha+\beta+2)}{\gamma+l-2}}\left(\rho_{\gamma+l-1}^n\right)^{\frac{\alpha+\beta+l-2}{\gamma+l-2}} + C_FC_{F,l}\int\limits_{\Omega}c_1^n.
\end{align*}
Then, for any $0\leq t\leq T$, if we integrate between $t$ and $T$ and use H\"older's inequality, we end up with
\begin{align*}
&\int\limits_{\Omega}\rho_l^n(T) + C_FC_{F,l}\int_{t}^T\int\limits_{\Omega}\rho_{\gamma+l-1}^n \\
&\quad \leq \int\limits_{\Omega}\rho_l^n(t) + 2C_QC_{Q,l}\int_{t}^T\int\limits_{\Omega}\left(\rho_{1}^n\right)^{\frac{2\gamma+l-(\alpha+\beta+2)}{\gamma+l-2}}\left(\rho_{\gamma+l-1}^n\right)^{\frac{\alpha+\beta+l-2}{\gamma+l-2}} + C_FC_{F,l}\int_{t}^T\int\limits_{\Omega}c_1^n\\
&\quad \leq \int\limits_{\Omega}\rho_l^n(t) + 2C_QC_{Q,l}\left\Vert\rho_{1}^n\right\Vert_{L^p(\Omega_T)}^{\frac{2\gamma+l-(\alpha+\beta+2)}{\gamma+l-2}}\left(\int_{t}^T\int\limits_{\Omega}\rho_{\gamma+l-1}^n\right)^{\frac{\alpha+\beta+l-2}{\gamma+l-2}} + C_FC_{F,l}\left\Vert\rho_{1}^n\right\Vert_{L^1(\Omega_T)},
\end{align*}
where $p=\frac{2\gamma+l-(\alpha+\beta+2)}{\gamma-(\alpha+\beta)}=1+\frac{\gamma+l-2}{\gamma-(\alpha+\beta)}$. But thanks to Proposition~\ref{th:first_moment_convergence}, we have that for any $p<\infty$, $\left(\rho_{1}^n\right)_n$ is bounded in $L^p(\Omega_T)$. Renaming the constants, we have shown that, for all $l>1$ also satisfying $l>2-(\gamma-\alpha)$ and $l>2-(\gamma-\beta)$, for all $n\in\N^*$ and all $0\leq t\leq T$,
\begin{equation}
\label{eq:rho_l_3}
\int\limits_{\Omega}\rho_l^n(T) + \tilde C_{1,l}\int_{t}^T\int\limits_{\Omega}\rho_{l+\gamma-1}^n \leq \int\limits_{\Omega}\rho_l^n(t) + \tilde C_{2,l}\left(\int_{t}^T\int\limits_{\Omega}\rho_{l+\gamma-1}^n\right)^{1-\theta_l} + \tilde C_{3,l},
\end{equation}
where $0<\theta_l<1$, $0<\tilde C_{j,l}<\infty$, $1\leq j\leq 3$, and crucially none of these constants depend on $n$. We are now ready to prove by induction that, for all $m\in\N^*$ and all $T>0$, there exists $\check C<\infty$ (depending on $k$, $m$, $\vert\Omega\vert$ and $T$) such that
\begin{equation}
\label{eq:rho_m}
\int_0^T\frac{t^{m-1}}{(m-1)!}\int\limits_{\Omega}\rho_{k+m(\gamma-1)}^n(t,x)dtdx \leq \check C,\quad \forall~n\in\N^*.
\end{equation}

We start by proving~\eqref{eq:rho_m} for $m=1$. Using \eqref{eq:rho_l_3} with $l=k$ and $t=0$ we get
\begin{equation*}
\tilde C_{1,k}\int_{0}^T\int\limits_{\Omega}\rho_{k+\gamma-1}^n \leq \int\limits_{\Omega}\rho_k^{in} + \tilde C_{3,k} + \tilde C_{2,k}\left(\int_{0}^T\int\limits_{\Omega}\rho_{k+\gamma-1}^n\right)^{1-\theta_k},
\end{equation*}
and Lemma~\ref{lem:elem1} then yields~\eqref{eq:rho_m} for $m=1$. We now take $m\in\N^*$ and consider~\eqref{eq:rho_l_3} with $l=k+m(\gamma-1)$. We get
\begin{align*}
\tilde C_{1,k+m(\gamma-1)}\int_{t}^T\int\limits_{\Omega}\rho_{k+(m+1)(\gamma-1)}^n &\leq \quad \int\limits_{\Omega}\rho_{k+m(\gamma-1)}^n(t) + \tilde C_{3,k+m(\gamma-1)}\\
&\quad + \tilde C_{2,k+m(\gamma-1)}\left(\int_{t}^T\int\limits_{\Omega}\rho_{k+(m+1)(\gamma-1)}^n\right)^{1-\theta_{k+m(\gamma-1)}}.
\end{align*}
Multiplying by $\frac{t^{m-1}}{(m-1)!}$, integrating for $t$ between $0$ and $T$, and assuming that~\eqref{eq:rho_m} holds for $m$, we end up with
\begin{align*}
&\tilde C_{1,k+m(\gamma-1)}\int_{0}^T\frac{t^{m-1}}{(m-1)!}\int_{t}^T\int\limits_{\Omega}\rho_{k+(m+1)(\gamma-1)}^n \leq \\
& \qquad \check C + \tilde C_{3,k+m(\gamma-1)}\frac{T^m}{m!} + \tilde C_{2,k+m(\gamma-1)}\int_{0}^T\frac{t^{m-1}}{(m-1)!}\left(\int_{t}^T\int\limits_{\Omega}\rho_{k+(m+1)(\gamma-1)}^n\right)^{1-\theta_{k+m(\gamma-1)}}.
\end{align*}
Lemma~\ref{lem:elem2} then yields a bound, that does not depend on $n$, for 
\begin{equation*}
\int_{0}^T\frac{t^{m-1}}{(m-1)!}\int_{t}^T\int\limits_{\Omega}\rho_{k+(m+1)(\gamma-1)}^n=\int_{0}^T\frac{t^{m}}{(m)!}\int\limits_{\Omega}\rho_{k+(m+1)(\gamma-1)}^n.
\end{equation*}
Thus \eqref{eq:rho_m} holds for $m+1$, and then for all $m\in\N^*$ by induction. 

Then, notice that~\eqref{eq:rho_m} with $m=1$ says exactly that the superlinear moment $\left(\rho_{k+\gamma-1}^n\right)_n$ is bounded in $L^1(\Omega_T)$. Therefore we can apply Proposition~\ref{prop:existence} case (ii) to extract from $\left(c^n\right)_n$ a weak solution $c$ of~\eqref{eq:syst_coag-frag}-\eqref{hyp:modele}. Besides, as we saw in the proof of Proposition~\ref{prop:existence} case (ii), the $L^1$ bound of the superlinear moment $\left(\rho_{k+\gamma-1}^n\right)_n$ allows us to prove that $\left(\rho_1^n\right)_n$ converges to $\rho_1$ in $L^1(\Omega_T)$. But for all $t\geq 0$ and all $n\in\N^*$ we have (rigorously)
\begin{equation*}
\int\limits_{\Omega}\rho_1^n(t) = \int\limits_{\Omega}\left(\rho_1^{in}\right)^n,
\end{equation*}
and since 
\begin{equation*}
\int\limits_{\Omega}\left(\rho_1^{in}\right)^n \underset{n\to\infty}{\longrightarrow}\int\limits_{\Omega}\rho_1^{in},
\end{equation*}
we get that, for almost all $t\geq 0$,
\begin{equation*}
\int\limits_{\Omega}\rho_1(t) = \int\limits_{\Omega}\rho_1^{in},
\end{equation*}
i.e. there is no gelation.

Finally, by Fatou's Lemma the bound~\eqref{eq:rho_m} is carried over to the moments of the limiting solution $c$, i.e. for all $m\in\N^*$
\begin{equation}
\label{eq:rho_m_final}
\int_0^T\frac{t^{m-1}}{(m-1)!}\int\limits_{\Omega}\rho_{k+m(\gamma-1)}(t,x)dtdx < \infty.
\end{equation}
\hfill $\qed$

\begin{remark}
\label{rem:assumption_th2}
The assumption~\eqref{hyp:convergence_di} on the convergence of the diffusion coefficients, the assumption $c^{in}_i\in L^{\infty}(\Omega)$ for all $i\geq 1$, and the assumption $N\leq 2$ were only used here to apply Proposition~\ref{th:first_moment_convergence} and get a bound on $\left(\rho_1^n\right)_n$ in $L^p(\Omega_T)$, for all $p<\infty$. Therefore, if we fix a $p>2$ such that $\rho_1^{in}\in L^{p}(\Omega)$, consider an arbitrary dimension $N$, assume~\eqref{hyp:closeness} instead of~\eqref{hyp:convergence_di} and remove the assumption $c^{in}_i\in L^{\infty}(\Omega)$ for all $i\geq 1$, we can now use Proposition~\ref{th:first_moment} to get a bound on $\left(\rho_1^n\right)_n$ in $L^p(\Omega_T)$ for this fixed $p$. The estimate~\eqref{eq:rho_l_3} then holds for all $l>1$ such that $1+\frac{\gamma+l-2}{\gamma-(\alpha+\beta)}\leq p$, and so do~\eqref{eq:rho_m} and~\eqref{eq:rho_m_final}, for all $m$ such that $1+\frac{\gamma+k+(m-1)(\gamma-1)-2}{\gamma-(\alpha+\beta)}\leq p$.
\end{remark}

\begin{remark}
\label{rem:beta}
As already pointed out, in Theorem~\ref{th:strong_frag} we made no assumption on the coefficients $\beta_{i,j}$ (aside from the microscopic mass conservation~\eqref{hyp:modele}). A fairly generic assumption one can add is the following:
\begin{equation}
\label{hyp:beta}
\forall~l>1,\ \exists C_{l}>0,\ \forall~i\in\N^*, \quad i^l-\sum_{j=1}^{i-1}j^l\beta_{i,j} \geq C_l i^{l}.
\end{equation}
A similar assumption is made (in the particular case of binary fragmentation) in~\cite{Cos95} for the homogeneous case.

While~\eqref{hyp:beta} may not hold for some coagulation-fragmentation models (it is for instance not satisfied for the Becker-D\"oring \cite{BalCarPen86} model), it does still hold for a broad range of models. Indeed, typical examples of fragmentation coefficients are
\begin{equation*}
B_i=i^{\gamma},\ \gamma\in\R \quad \text{and}\quad \beta_{i,j}=\frac{i}{\sum_{k=1}^{i-1}k^{1+\nu}}j^{\nu}, \nu>-2,
\end{equation*}
see \cite{LauMis02} and the references therein. For such coefficients, one can check that~\eqref{hyp:beta} is always satisfied.

Assuming~\eqref{hyp:beta}, we can then use it instead of~\eqref{eq:borne_beta} in the proof of Theorem~\ref{th:strong_frag}, and gain one power of $i$. The moment of higher order that appears from the fragmentation term is then $\rho_{l+\gamma}$ (in place of $\rho_{l+\gamma-1}$), and the rest of the proof still holds if we only assume $\alpha+\beta<\gamma+1$ and $\gamma\geq 0$ (instead of $\alpha+\beta<\gamma$ and $\gamma\geq 1$).
\end{remark}

\section{Propagation of smoothness}
\label{sec:smooth}

This section is devoted to the proof of Theorem~\ref{th:all_moments_w}. The main argument is to show that a control of all moments $\rho_k$, $k\geq 0$, in all spaces $L^p(\Omega_T)$, $p<\infty$, allows for propagation of smoothness (Proposition~\ref{prop:smooth}). We begin with a lemma to control the coagulation and fragmentation terms $Q_i$ and $F_i$. 

\begin{lemma}\label{lem:ci_to_Fi}
Let $\Omega$ be a smooth bounded domain of $\R^N$ and $s\in\N$. Assume that $(c_i)_{i\in\N^*}$ is a sequence of positive functions defined on $\Omega_T$ such that
\begin{equation}
\label{hyp:ikc_i}
\sup\limits_{i\in\N^*} \left\Vert i^k c_i \right\Vert_{W^{s,p}(\Omega_T)} < +\infty,\quad \forall k\geq 0,\ \forall p<+\infty.
\end{equation}
Assume also~\eqref{def:Q}-\eqref{hyp:modele}. For the coagulation and fragmentation coefficients, assume that there exists $C\geq 0$ and $\gamma_{max}\in\R$ such that, for all $i,j\in\N^*$
\begin{equation}
\label{hyp:lem_QF}
a_{i,j}\leq Cij \quad \text{and}\quad B_i\leq C i^{\gamma_{\max}}.
\end{equation}
\par
Then, the following estimates hold
\begin{equation}
\label{eq:ikQ_i}
\sup\limits_{i\in\N^*} \left\Vert i^k Q_i(c) \right\Vert_{W^{s,p}(\Omega_T)} < +\infty,\quad \forall k\geq 0,\ \forall p<+\infty.
\end{equation}
and
\begin{equation}
\label{eq:ikF_i}
\sup\limits_{i\in\N^*} \left\Vert i^k F_i(c) \right\Vert_{W^{s,p}(\Omega_T)} < +\infty,\quad \forall k\geq 0,\ \forall p<+\infty.
\end{equation}
\end{lemma}
\begin{proof}
The bound~\eqref{eq:ikQ_i} for the coagulation term was already proven in~\cite{BreDevFel16}. To get the bound~\eqref{eq:ikF_i} for the fragmentation term, we estimate (remembering~\eqref{def:F})
\begin{align*}
\left\Vert i^k F_i(c) \right\Vert_{W^{s,p}(\Omega_T)} &\leq \sum_{j=1}^{\infty}B_{i+j}\beta_{i+j,i} i^k\left\Vert c_{i+j}\right\Vert_{W^{s,p}(\Omega_T)} + B_i i^k\left\Vert c_i \right\Vert_{W^{s,p}(\Omega_T)}.
\end{align*}
But by~\eqref{hyp:modele}, $\beta_{i+j,i}\leq\frac{i+j}{i}$, and using~\eqref{hyp:lem_QF} we end up with
\begin{align*}
\left\Vert i^k F_i(c) \right\Vert_{W^{s,p}(\Omega_T)} &\leq C\sum_{j=1}^{\infty}(i+j)^{\gamma_{max}+1}i^{k-1}\left\Vert c_{i+j}\right\Vert_{W^{s,p}(\Omega_T)} + Ci^{\gamma_{max}+k}\left\Vert c_i \right\Vert_{W^{s,p}(\Omega_T)} \\
&\leq C\sum_{j=0}^{\infty}\left\Vert (i+j)^{\gamma_{max}+k} c_{i+j}\right\Vert_{W^{s,p}(\Omega_T)} \\
&\leq C\sup_{i\in\N^*}\left\Vert i^{\gamma_{max}+k+2} c_{i}\right\Vert_{W^{s,p}(\Omega_T)} \sum_{j=0}^{\infty} \frac{1}{(i+j)^2} \\
&\leq C\sup_{i\in\N^*}\left\Vert i^{\gamma_{max}+k+2} c_{i}\right\Vert_{W^{s,p}(\Omega_T)} \sum_{j=0}^{\infty} \frac{1}{(1+j)^2}
\end{align*}
and this bound is finite by~\eqref{hyp:ikc_i}.
\end{proof}

As already stated in~\cite{BreDevFel16}, Lemma~\ref{lem:ci_to_Fi} can then be used to get $W^{s,p}$ estimates for solutions of~\eqref{eq:syst_coag-frag}-\eqref{hyp:modele}.
\begin{proposition}
\label{prop:smooth}
Let $\Omega$ be a smooth bounded domain of $\R^N$, and $T>0$. Assume~\eqref{hyp:lem_QF} for the coagulation and fragmentation coefficients and~\eqref{hyp:di} for the diffusion coefficients. Let $c$ be a solution of~\eqref{eq:syst_coag-frag}-\eqref{hyp:modele} such that $\rho_k$ lies in $L^p(\Omega_T)$, for all $k\geq 0$ and all $p<\infty$. Then $\rho_k$ lies in the Sobolev space $W^{s,p}(\Omega_T)$, for all $k\geq 0$, all $s\in\N$ and all $p<\infty$.
\end{proposition}
\begin{proof}
We are going to show, by induction on $s$, that for all $s\in\N$,
\begin{equation}
\label{eq:induc_s}
\sup\limits_{i\geq 1} \left\Vert i^k c_i \right\Vert_{W^{s,p}(\Omega_T)} < +\infty,\quad \forall k\in\N,\ \forall p<+\infty.
\end{equation}
For all $i\in\N^*$, we have $\left\Vert i^k c_i \right\Vert_{L^{p}(\Omega_T)}\leq \left\Vert \rho_k \right\Vert_{L^{p}(\Omega_T)}$, therefore~\eqref{eq:induc_s} holds for $s=0$. Assuming~\eqref{eq:induc_s} for a given $s\in\N$, Lemma~\ref{lem:ci_to_Fi} shows that
\begin{equation*}
\sup\limits_{i\geq 1} \left\Vert i^k \left(Q_i(c)+F_i(c)\right) \right\Vert_{W^{s,p}(\Omega_T)} < +\infty,\quad \forall k\in\N,\ \forall p<+\infty.
\end{equation*}
Since 
\begin{equation*}
\left(\partial_t - d_i \Delta_x\right) i^kc_i = i^k\left(Q_i(c)+F_i(c)\right),
\end{equation*}
the regularizing properties of the heat equation yield that
\begin{equation*}
\sup\limits_{i\geq 1} \left\Vert i^k c_i \right\Vert_{W^{s+1,p}(\Omega_T)} < +\infty,\quad \forall k\in\N,\ \forall p<+\infty,
\end{equation*}
where we also used~\eqref{hyp:di}, i.e. that the $d_i$ are bounded above and below, which ensure that the regularity estimates are uniform w.r.t. $i$. Therefore~\eqref{eq:induc_s} holds for all $s\in\N$. Notice that we also get $W^{s,p}$ estimates for moments $\rho_k$ of any order, because
\begin{equation*}
\left\Vert \rho_k \right\Vert_{W^{s,p}(\Omega_T)} \leq \sum_{i=1}^{\infty} \frac{1}{i^2} \left\Vert i^{k+2}c_i\right\Vert_{W^{s,p}(\Omega_T)} \leq \sup\limits_{i\geq 1} \left\Vert i^{k+2}c_i\right\Vert_{W^{s,p}(\Omega_T)}\sum_{i=1}^{\infty} \frac{1}{i^2}.
\end{equation*} 
\end{proof}

Finally, we can give the
\newline\noindent\textit{Proof of Theorem~\ref{th:all_moments_w}.}
We want to apply Proposition~\ref{prop:smooth}. Notice that~\eqref{hyp:lem_QF} is implied by the assumptions of Theorem~\ref{th:all_moments_w}. 

In the sublinear coagulation case, the fact that $\rho_k$ lies in $L^p(\Omega_T)$ for all $k\geq 0$ and all $p<\infty$ was already proven in \cite[Theorem 1.9]{BreDevFel16}, in the case of pure coagulation ($F_i=0$ for all $i\in\N^*$). The proof is a more technical version of the one of Proposition~\ref{th:first_moment_convergence} that can be readily extended to cases including fragmentation. One only needs to assume $N\leq 2$ (to get Lemma~\ref{lem:c_i}, which is valid in any dimension only when there is no fragmentation). Then, since the contribution of the fragmentation to the evolution of truncated moment $\rho_k^{I,n}$ of any order is non positive (thanks to~\eqref{hyp:modele}):
\begin{equation*}
\sum_{i=I}^{\infty}i^kF_i = -\sum_{i=I}^{\infty}B_ic_i\left(i^k-\sum_{j=I}^{i-1}j^k\beta_{i,j}\right) \leq 0,
\end{equation*}
the proof from~\cite{BreDevFel16} still holds without further modifications, even when the fragmentation is nonzero, and we get that $\rho_k$ lies in $L^p(\Omega_T)$ for all $k\geq 0$ and all $p<\infty$, in the sublinear coagulation case.

In the strong fragmentation case, since the assumptions implies that $\rho_k^{in}$ lies in $L^1(\Omega)$ for all $k\geq 0$, Theorem~\ref{th:strong_frag} yields that $\rho_k$ lies in $L^1(\Omega_T)$ for all $k\geq 0$. But by Proposition~\ref{th:first_moment_convergence}, we also have that $\rho_1$ lies in $L^p(\Omega_T)$ for all $p<\infty$. By interpolation we then get that $\rho_k$ lies in $L^p(\Omega_T)$ for all $k\geq 0$ and all $p<\infty$.

Therefore, we can use Proposition~\ref{prop:smooth} in both the sublinear coagulation and the strong fragmentation cases. The $\C^{\infty}$ regularity announced in Theorem~\ref{th:all_moments_w} is then a straightforward consequence of Sobolev embeddings.  

We finish with the proof of the uniqueness statement, which is an extension from a result in~\cite{HamRez07} where only the case without fragmentation is treated.
We consider $c=\left(c_i\right)_{i\in\N^*}$ and $\tilde c=\left(\tilde c_i\right)_{i\in\N^*}$ two smooth solutions to the coagulation-fragmentation system \eqref{eq:syst_coag-frag}-\eqref{hyp:modele} and compute
\begin{align*}
\frac{d}{dt}\int\limits_{\Omega} \sum_{i=1}^{\infty}i\vert c_i-\tilde c_i\vert &= \int\limits_{\Omega} \sum_{i=1}^{\infty}\varphi_i \left(Q_i(c)-Q_i(\tilde c)+F_i(c)-F_i(\tilde c)\right) \\
&= \frac{1}{2}\int\limits_{\Omega} \sum_{i=1}^{\infty}\sum_{j=1}^{\infty}a_{i,j}(c_ic_j-\tilde c_i \tilde c_j)(\varphi_{i+j}-\varphi_i-\varphi_j) \\
&\quad - \int\limits_{\Omega} \sum_{i=2}^{\infty}B_i(c_i-\tilde c_i)\left(\varphi_i-\sum_{j=1}^{i-1}\beta_{i,j}\varphi_j\right),
\end{align*}
where
\begin{equation*}
\varphi_i=i \sgn (c_i-\tilde c_i).
\end{equation*}
We first bound the first term. Rewriting $(c_ic_j-\tilde c_i \tilde c_j)$ as $(c_i-\tilde c_i)c_j + (c_j-\tilde c_j)\tilde c_i$, estimating
\begin{align*}
(c_i-\tilde c_i)c_j(\varphi_{i+j}-\varphi_i-\varphi_j) &\leq (i+j)\vert c_i-\tilde c_i\vert c_j - i\vert c_i-\tilde c_i\vert c_j + j\vert c_i-\tilde c_i\vert c_j \\
& = 2j\vert c_i-\tilde c_i\vert c_j,
\end{align*}
and using the symmetry of the coagulation coefficients, we end up with
\begin{equation*}
\frac{1}{2}\int\limits_{\Omega} \sum_{i=1}^{\infty}\sum_{j=1}^{\infty}a_{i,j}(c_ic_j-\tilde c_i \tilde c_j)(\varphi_{i+j}-\varphi_i-\varphi_j) \leq \int\limits_{\Omega} \sum_{i=1}^{\infty}\sum_{j=1}^{\infty}a_{i,j}\vert c_i-\tilde c_i\vert j(c_j+\tilde c_j).
\end{equation*}
Since $a_{i,j}\leq Cij$, we obtain
\begin{align*}
& \frac{1}{2}\int\limits_{\Omega} \sum_{i=1}^{\infty}\sum_{j=1}^{\infty}a_{i,j}(c_ic_j-\tilde c_i \tilde c_j)(\varphi_{i+j}-\varphi_i-\varphi_j)\\
&\leq C\int\limits_{\Omega} \sum_{i=1}^{\infty}i\vert c_i-\tilde c_i\vert\sum_{j=1}^{\infty}j^2(c_j+\tilde c_j) \\
&\leq C\left(\left\Vert \rho_2(c)\right\Vert_{L^{\infty}(\Omega_T)} + \left\Vert \rho_2(\tilde c)\right\Vert_{L^{\infty}(\Omega_T)}\right)\int\limits_{\Omega} \sum_{i=1}^{\infty}i\vert c_i-\tilde c_i\vert.
\end{align*}
For the second term, we have (using~\eqref{hyp:modele})
\begin{align*}
- \int\limits_{\Omega} \sum_{i=2}^{\infty}B_i(c_i-\tilde c_i)\left(\varphi_i-\sum_{j=1}^{i-1}\beta_{i,j}\varphi_j\right) &\leq \int\limits_{\Omega} \sum_{i=2}^{\infty}B_i\vert c_i-\tilde c_i\vert \left(i+\sum_{j=1}^{i-1}\beta_{i,j}j\right) \\
&= 2\int\limits_{\Omega} \sum_{i=2}^{\infty}B_i i\vert c_i-\tilde c_i\vert \\
&\leq 2\sup\limits_{i\in\N^*} B_i \int\limits_{\Omega} \sum_{i=2}^{\infty}i\vert c_i-\tilde c_i\vert.
\end{align*}
Putting everything back together, we get
\begin{equation*}
\frac{d}{dt}\int\limits_{\Omega} \sum_{i=1}^{\infty}i\vert c_i-\tilde c_i\vert \leq \left( C\left(\left\Vert \rho_2(c)\right\Vert_{L^{\infty}(\Omega_T)} + \left\Vert \rho_2(\tilde c)\right\Vert_{L^{\infty}(\Omega_T)}\right) + 2\sup\limits_{i\in\N^*} B_i \right) \int\limits_{\Omega} \sum_{i=1}^{\infty}i\vert c_i-\tilde c_i\vert,
\end{equation*}
and by Gronwall's Lemma we conclude that, if the solutions $c$ and $\tilde c$ have the same initial data, then they remain equal for all positive time. \hfill $\qed$

\section*{Acknowledgement}

It is a pleasure to thank Laurent Desvillettes for helpful discussions during the preparation of this work. The research leading to this paper was partially funded by the french ``ANR blanche'' project Kibord: ANR-13-BS01-0004.

\end{document}